\documentclass[11pt,a4paper]{article}
\usepackage[T1]{fontenc}
\usepackage[utf8]{inputenc}
\usepackage{lmodern,amsmath,amssymb,amsthm,mathtools,microtype}
\usepackage[margin=27mm,headheight=14pt]{geometry}
\usepackage{booktabs,enumitem,tikz,fancyhdr,needspace,aliascnt,xurl}
\usepackage[numbers,sort&compress]{natbib}
\usepackage[hidelinks]{hyperref}
\usepackage[capitalise,nameinlink,noabbrev]{cleveref}
\usetikzlibrary{arrows.meta}
\hypersetup{pdftitle={A forbidden-induced-subgraph characterization of beautiful graphs},
 pdfauthor={Henning Wunderlich},
 pdfsubject={Revised draft: forbidden families for beautiful Berge graphs; double-bound and biclique representations}}
\setlist[enumerate]{itemsep=3pt,topsep=5pt}
\newtheorem{theorem}{Theorem}[section]
\newaliascnt{lemma}{theorem}
\newtheorem{lemma}[lemma]{Lemma}
\aliascntresetthe{lemma}
\newaliascnt{proposition}{theorem}
\newtheorem{proposition}[proposition]{Proposition}
\aliascntresetthe{proposition}
\newaliascnt{corollary}{theorem}
\newtheorem{corollary}[corollary]{Corollary}
\aliascntresetthe{corollary}
\theoremstyle{definition}
\newaliascnt{definition}{theorem}

\aliascntresetthe{definition}
\newaliascnt{example}{theorem}
\newtheorem{example}[example]{Example}
\aliascntresetthe{example}
\theoremstyle{remark}
\newaliascnt{remark}{theorem}

\aliascntresetthe{remark}
\newcommand{\F}{\mathcal F}
\newcommand{\cs}{\mathcal G}
\newcommand{\csg}{\mathsf{CS}}
\newcommand{\beautiful}{\mathsf{Beautiful}}
\newcommand{\Comp}{\operatorname{Comp}}
\newcommand{\DB}{\operatorname{DB}}
\newcommand{\HDB}{\mathsf{HDB}}
\newcommand{\Forb}{\operatorname{Forb}_{\mathrm{ind}}}
\newcommand{\gem}{\mathrm{gem}}
\newcommand{\net}{\mathrm{net}}
\newcommand{\watch}{\mathrm{watch}}

\newcommand{\down}{\mathord\downarrow}
\newcommand{\up}{\mathord\uparrow}
\newcommand{\dotcup}{\mathbin{\dot\cup}}
\newcommand{\PaperAuthor}{Henning Wunderlich}
\title{\bfseries A forbidden-induced-subgraph\\characterization of beautiful graphs}
\author{\PaperAuthor\thanks{Offenbach am Main, Germany. Email: \texttt{henningwunderlich@t-online.de}}}
\date{5 October 2026}

\begin{document}
\maketitle
\begin{abstract}
A graph is beautiful if each of its induced subgraphs is the intersection
graph of all maximal nonempty 1-rectangles of a binary matrix, with
adjacency defined by a common cell. Beautiful graphs were introduced as
a hereditary class of Berge graphs, but a complete forbidden-induced-subgraph
characterization was not obtained. We prove that a finite graph is beautiful
if and only if it has no induced $C_4$, gem, net, watch, or odd hole.
More precisely, these graphs are exactly the $C_4$-free comparability graphs
admitting a partial order in which every interval is a chain. The proof
constructs such an order from inclusion-maximal closed neighbourhoods:
their representatives induce a bipartite graph whose domination regions
admit compatible orientations. One order matrix then represents the graph
and, through its principal submatrices, every induced subgraph. The
representation step is formulated using classical double-bound graphs and
the established correspondence between maximal bicliques and
interval-intersection-closed posets. Consequences include polynomial-time
recognition, the exact minimal obstruction families, and a corrected
characterization in the $K_4$-free case.
\end{abstract}

\medskip
\noindent\textbf{Mathematics Subject Classification (2020).}
Primary 05C75; Secondary 06A06.

\smallskip
\noindent\textbf{Keywords.}
Cover-structure graph; beautiful graph; forbidden induced subgraph; Berge graph; comparability graph; double-bound graph; maximal rectangle.

\section{Introduction and main result}\label{sec:intro}
Let $M\colon X\times Y\to\{0,1\}$ be a finite binary matrix. A
\emph{nonempty 1-rectangle} is a product $A\times B$, with
$\varnothing\ne A\subseteq X$ and $\varnothing\ne B\subseteq Y$, on which all
entries of $M$ are one. Its maximality is with respect to inclusion; equivalently,
neither a row nor a column can be added while preserving this property.
The \emph{cover-structure graph}, or \emph{cs-graph}, $\cs(M)$ has as vertices
\emph{all} maximal nonempty 1-rectangles of $M$. Two distinct vertices
$A\times B$ and $A'\times B'$ are adjacent precisely when
\begin{equation}\label{eq:cell-intersection}
 A\cap A'\ne\varnothing\quad\text{and}\quad B\cap B'\ne\varnothing.
\end{equation}
Thus an intersection means a common matrix cell, not merely a common row or
column. We write $\csg$ for the class of graphs isomorphic to some $\cs(M)$.

A graph is \emph{beautiful} if every induced subgraph, including the graph
itself, belongs to $\csg$. The class $\beautiful$ is therefore the largest
hereditary subclass of $\csg$. These definitions use the one-colour convention
of \citet{Wunderlich2009}. Requiring all maximal rectangles is essential:
deleting vertices of $\cs(M)$ need not leave the complete maximal-rectangle
family of any matrix. Consequently $\csg$ itself need not be hereditary.

\citet{Wunderlich2009} introduced beautiful graphs as a class of Berge
graphs, compared them with familiar Berge subclasses, and established
necessary forbidden induced subgraphs and sufficient conditions on
restricted classes. Here a Berge graph has no odd hole or odd antihole.
The characterization question concerns the exact identity of this
hereditary class, not merely its containment in the Berge graphs. This is
the graph-class viewpoint developed in surveys and inclusion catalogues
such as \cite{BrandstadtLeSpinrad1999,ISGCI}.

Our main contribution is a complete forbidden-induced-subgraph
classification of the beautiful class introduced in
\cite{Wunderlich2008,Wunderlich2009}. We prove that the previously identified
obstructions are sufficient, with no additional families and no bound on
the clique number. Set
\begin{equation}\label{eq:family}
 \F=\{C_4,\gem,\net,\watch\}\ \cup\ \{C_{2k+1}:k\ge2\}.
\end{equation}
\Needspace{6\baselineskip}
The gem is a $P_4$ together with a universal vertex. The net is a triangle
with a pendant leaf at each of its three vertices. The watch is a diamond
$K_4-e$ with a pendant leaf at each of its two degree-three vertices; see
\cref{fig:obstructions}. The net is the graph called ``star'' in
\cite{Wunderlich2008,Wunderlich2009}; it is not the usual star $K_{1,3}$. A hole is an induced cycle of length at least four;
an odd hole therefore has length at least five.

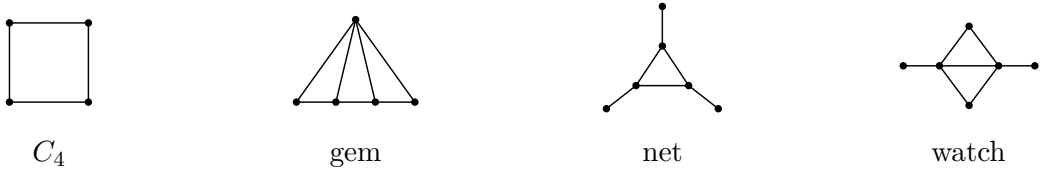
\begin{figure}[tb]
\centering
\begin{minipage}[b]{.22\textwidth}\centering
\begin{tikzpicture}[scale=.87,line width=.55pt]
\path[use as bounding box] (-1.05,-.85) rectangle (1.05,1.05);
\coordinate (a) at (-.6,-.5);\coordinate (b) at (.6,-.5);
\coordinate (c) at (.6,.7);\coordinate (d) at (-.6,.7);
\draw (a)--(b)--(c)--(d)--cycle;
\foreach \v in {a,b,c,d}\fill (\v) circle (1.6pt);
\end{tikzpicture}\par $C_4$
\end{minipage}\hfill
\begin{minipage}[b]{.22\textwidth}\centering
\begin{tikzpicture}[scale=.87,line width=.55pt]
\path[use as bounding box] (-1.05,-.85) rectangle (1.05,1.05);
\coordinate (a) at (-.9,-.5);\coordinate (b) at (-.3,-.5);
\coordinate (c) at (.3,-.5);\coordinate (d) at (.9,-.5);
\coordinate (u) at (0,.75);
\draw (a)--(b)--(c)--(d);
\foreach \v in {a,b,c,d}\draw (u)--(\v);
\foreach \v in {a,b,c,d,u}\fill (\v) circle (1.6pt);
\end{tikzpicture}\par gem
\end{minipage}\hfill
\begin{minipage}[b]{.22\textwidth}\centering
\begin{tikzpicture}[scale=.87,line width=.55pt]
\path[use as bounding box] (-1.05,-.85) rectangle (1.05,1.05);
\coordinate (a) at (0,.35);\coordinate (b) at (-.4,-.25);
\coordinate (c) at (.4,-.25);\coordinate (x) at (0,.95);
\coordinate (y) at (-.85,-.6);\coordinate (z) at (.85,-.6);
\draw (a)--(b)--(c)--cycle (a)--(x) (b)--(y) (c)--(z);
\foreach \v in {a,b,c,x,y,z}\fill (\v) circle (1.6pt);
\end{tikzpicture}\par net
\end{minipage}\hfill
\begin{minipage}[b]{.22\textwidth}\centering
\begin{tikzpicture}[scale=.87,line width=.55pt]
\path[use as bounding box] (-1.05,-.85) rectangle (1.05,1.05);
\coordinate (u) at (-.45,.05);\coordinate (v) at (.45,.05);
\coordinate (a) at (0,.65);\coordinate (b) at (0,-.55);
\coordinate (x) at (-1,.05);\coordinate (y) at (1,.05);
\draw (u)--(a)--(v)--(b)--cycle (u)--(v) (u)--(x) (v)--(y);
\foreach \v in {u,v,a,b,x,y}\fill (\v) circle (1.6pt);
\end{tikzpicture}\par watch
\end{minipage}
\caption{The four individual obstructions. The remaining forbidden family consists
of all odd holes $C_5,C_7,C_9,\ldots$. Every exclusion is induced.}
\label{fig:obstructions}
\end{figure}

For a poset $P$, its \emph{comparability graph} $\Comp(P)$ joins two distinct
elements whenever they are comparable. We say that $P$ has \emph{chain
intervals} if every interval
\[
 [a,b]_P=\{x:a\le_Px\le_Pb\},\qquad a\le_Pb,
\]
is a chain. Equivalently, incomparable elements never have both a common
lower bound and a common upper bound.

\begin{theorem}[Beautiful-graph characterization]\label{thm:main}
For every finite graph $G$, the following statements are equivalent.
\begin{enumerate}[label=\textnormal{(\roman*)}]
\item $G$ is beautiful.
\item $G$ has no induced member of $\F$.
\item $G$ is $C_4$-free and $G\cong\Comp(P)$ for a poset $P$ with chain intervals.
\item $G$ is Berge and has no induced $C_4$, gem, net, or watch.
\end{enumerate}
In \textnormal{(iii)}, $P$ can be constructed on $V(G)$. Its order matrix
\begin{equation}\label{eq:order-matrix}
 M_P(x,y)=\begin{cases}1,&x\le_Py,\\0,&\text{otherwise}\end{cases}
\end{equation}
has exactly $|V(G)|$ maximal nonempty 1-rectangles, naturally indexed by
$V(G)$, and simultaneously satisfies
\begin{equation}\label{eq:hereditary-matrix}
 \cs\bigl(M_P[W,W]\bigr)\cong G[W]
 \qquad\text{for every }W\subseteq V(G).
\end{equation}
\end{theorem}

In particular,
\[
 \beautiful=\Forb\bigl(C_4,\gem,\net,\watch,C_5,C_7,C_9,\ldots\bigr).
\]
Equivalently, within the class of Berge graphs the classification has four
individual obstructions:
\begin{equation}\label{eq:berge-classification}
 \beautiful=\mathsf{Berge}\cap\Forb(C_4,\gem,\net,\watch).
\end{equation}
There is no restriction on the clique number. The stronger representation
\eqref{eq:hereditary-matrix} makes the hereditary conclusion explicit, without
constructing a separate unrelated matrix for each induced subgraph.

\paragraph{Proof outline.}
We first establish an order-matrix normal form for arbitrary cs-graphs.
It identifies a sufficient hereditary certificate: a $C_4$-free comparability
graph represented by a poset with chain intervals. The main construction
then uses representatives of the inclusion-maximal closed neighbourhoods of
an $\F$-free graph. These representatives induce a bipartite graph. Each
vertex is dominated by one or two representatives, giving a partition into
regions with controlled interfaces. The bipartition determines compatible
local orientations, yielding the required poset.

\subsection{Relation to earlier work}\label{sec:related}
The classification in \cref{thm:main} is distinct from the older
representation theory used in its proof. In particular, neither the
poset-graph operator nor the order-matrix construction below is introduced
as a new framework. The new sufficiency argument is the
maximal-neighbourhood decomposition and its compatible global orientation
in \cref{sec:structure,sec:orientation}.

\paragraph{Double-bound graphs.}
The graph on a poset $P$ that joins two elements with both a common lower
bound and a common upper bound is its \emph{double-bound graph}, denoted
$\DB(P)$. This is the operator of McMorris and Zaslavsky
\cite{McMorrisZaslavsky1982}. Their general double-bound characterization
was corrected by Diny \cite{Diny1985}; we do not invoke that
characterization. Subsequent work considers transformations preserving
bound graphs \cite{EraOgawaTsuchiya2001}, clique-cover representations
\cite{IwaiOgawaTsuchiya2003Clique}, and constructions reducing double-bound
graphs to bipartite graphs and rebuilding them by additions
\cite{IwaiOgawaTsuchiya2003Construction}. Other studies relate chordal
bound graphs and restricted double-bound classes to forbidden subposets
\cite{IwaiOgawaTsuchiya2008,OgawaTagusariTsuchiya2009}.

Hereditary double-bound graphs were studied before the introduction of
cover-structure graphs
\cite{EraIwaiOgawaTsuchiya2003,EraOgawaTsuchiya2008}.
They provide a close predecessor to the hereditary representation question,
but are not the same class as beautiful graphs. In \cref{prop:hdb} we
prove directly that beautiful graphs form a proper subclass of the graphs
all of whose induced subgraphs are double-bound graphs; $C_4$ separates
the two classes. We do not use a characterization of hereditary
double-bound graphs to derive \cref{thm:main}.

\paragraph{Bicliques and formal concepts.}
Let $B_M$ be the bipartite incidence graph of $M$, with parts $X,Y$ and
an edge $xy$ exactly when $M(x,y)=1$. Its maximal bicliques with two
nonempty parts are precisely the maximal 1-rectangles. Sharing a matrix
cell is exactly sharing an edge of $B_M$. In the notation of the
edge-biclique operator studied by Montero and Legay
\cite{MonteroLegay2022},
\begin{equation}\label{eq:edge-biclique}
 \cs(M)=KB_e(B_M),\qquad
 \csg=KB_e(\mathsf{Bipartite}).
\end{equation}
This must not be confused with the ordinary biclique operator $KB$, for
which one common vertex suffices. The row and column sets of a maximal
rectangle are, respectively, the extent and intent of a formal concept;
the restriction here is that both are nonempty. Their inclusion order is
therefore part of the standard framework of formal concept analysis
\cite{GanterWille1999}.

Groshaus and Guedes \cite{GroshausGuedes2021} established the
interval-intersection-closed (IIC) order on maximal bicliques of a
bipartite graph and the converse predecessors--successors construction.
In particular, their IIC construction has exactly the principal maximal
bicliques. Taking \emph{edge} intersections gives the normal form
\[
 \csg=\{\DB(P):P\text{ is a finite IIC poset}\}.
\]
\Cref{prop:normal} formulates this consequence in cs-graph language and
includes an elementary proof. This normal form supplies the representation step in the proof of the
hereditary classification.
Recent work on biclique-part inclusion graphs continues the connection
with comparability and IIC-type orders \cite{CruzGroshausGuedes2025};
\citet{Groshaus2025Survey} surveys biclique recognition and inverse problems.

\paragraph{Hereditary representations and root reconstruction.}
Montero's vertex-removal problem for ordinary biclique graphs
\cite{Montero2022Removal} addresses a closely related difficulty: after
deleting an output vertex, can a new root be constructed that has exactly
the remaining maximal objects? Groshaus, Hell and Stacho
\cite{GroshausHellStacho2012} study hereditary and Helly properties of
biclique edge-set structures, including a forbidden-subgraph
characterization for hereditary biclique line graphs. Their biclique line
graph has root \emph{edges}, rather than maximal bicliques, as vertices;
its characterization is not an edge-biclique characterization.
Apollonio, Caramia and Franciosa \cite{ApollonioCaramiaFranciosa2015}
relate bipartite distance-hereditary roots to tree-shaped Galois orders.
That restriction is on the root graph and its order, whereas beautifulness
is hereditary membership of the output graph in \eqref{eq:edge-biclique}.

\paragraph{Orders with chain intervals.}
The chain-interval condition is also established terminology in other
settings: it excludes a diamond as an induced (or full) \emph{subposet}.
For a finite poset it is equivalent to the Hasse diagram having at most
one directed path between each ordered pair. These Hasse diagrams are
multitrees in the sense of Furnas and Zacks \cite{FurnasZacks1994}.
The same condition appears in the theory of incidence algebras
\cite[Lemmas~6.7 and~6.10]{FloresKimuraRognerud2022}; algebraic
hereditariness there has a different meaning. It does not require the
comparability graph to be diamond-free, and does not require the
underlying undirected Hasse diagram to be a forest. The contribution here
is to construct such an order from the precise forbidden graph families
in \eqref{eq:family} and to make its order-matrix representation hereditary.

\paragraph{Conventions.}
All graphs and posets are finite, and graphs are simple. Throughout,
``$F$-free'' means free of \emph{induced} copies of $F$. We write
$N[v]=N(v)\cup\{v\}$ for a closed neighbourhood. Distinct vertices with equal
closed neighbourhoods are \emph{true twins}. A vertex is \emph{simplicial}
if its neighbourhood is a clique. The empty graph is represented by a matrix
with no 1-entry; for principal-submatrix statements we allow the empty
$0\times0$ matrix. No connectedness assumption is made.

\section{Double-bound graphs and the cs normal form}\label{sec:normal}
For a poset $P$, put $\down p=\{x:x\le_Pp\}$ and
$\up p=\{y:p\le_Py\}$. Following \citet{GroshausGuedes2021}, call $P$
\emph{interval-intersection-closed} (IIC) if every nonempty intersection of
two principal ideals is a principal ideal, and dually for principal filters.
Equivalently, a pair with a common lower bound has an infimum, and a pair
with a common upper bound has a supremum. Empty intersections are allowed;
there need not be a least or a greatest element.

Recall the classical double-bound graph $\DB(P)$
\cite{McMorrisZaslavsky1982}, defined on $P$ by
\begin{equation}\label{eq:boundgraph}
 xy\in E(\DB(P))
 \quad\Longleftrightarrow\quad
 (\down x)\cap(\down y)\ne\varnothing
 \ \text{ and }\
 (\up x)\cap(\up y)\ne\varnothing,
 \qquad x\ne y.
\end{equation}
In general $\DB(P)$ is not the comparability graph: incomparable elements can
have both kinds of common bound.

The next statement is the edge-intersection consequence of the IIC
biclique representations of \citet{GroshausGuedes2021}. The two structural
ingredients---IIC of the biclique order and completeness of the principal
bicliques in the converse construction---are Lemmas~5, 7, and~8 of the
2020 preprint cited with that paper. We include the proof to keep the
matrix conventions and the requirement of \emph{all} maximal rectangles
explicit.

\begin{proposition}[Cs normal form]\label{prop:normal}
A finite graph $G$ is a cs-graph if and only if $G\cong \DB(P)$ for an IIC
poset on $V(G)$. For every IIC poset $P$, its order matrix satisfies
$\cs(M_P)\cong\DB(P)$ and has exactly the maximal nonempty 1-rectangles
\begin{equation}\label{eq:principal-rectangles}
 R_p=(\down p)\times(\up p),\qquad p\in P.
\end{equation}
\end{proposition}
\begin{proof}
Let $G=\cs(M)$. Order its maximal rectangles by their row sets:
\[
 A\times B\ \le\ A'\times B'\quad\Longleftrightarrow\quad A\subseteq A'.
\]
Maximality says that $A$ is the set of all rows equal to one on $B$, and
$B$ is the set of all columns equal to one on $A$. Consequently the order
is also characterized by $B'\subseteq B$, and equal row sets determine equal
rectangles.

For two rectangles $A_1\times B_1$ and $A_2\times B_2$, suppose that
$A=A_1\cap A_2$ is nonempty. Let $B$ be the set of columns equal to one on
$A$. It contains $B_1\cup B_2$ and is nonempty. A row equal to one on $B$
must belong to both $A_1$ and $A_2$, while every row of $A$ is equal to one
on $B$. Thus $A\times B$ is maximal and is the infimum of the two rectangles.
If their row sets are disjoint, they have no common lower bound. The dual
argument, using $B_1\cap B_2$, treats upper bounds and suprema. This proves
that the rectangle poset $P$ is IIC. It also shows that the two nonemptiness
conditions in \eqref{eq:boundgraph} are exactly
\eqref{eq:cell-intersection}, so $G=\DB(P)$.

Conversely, let $P$ be IIC. Each $R_p$ in \eqref{eq:principal-rectangles}
is a 1-rectangle of $M_P$. Its column $p$ excludes every row outside
$\down p$, and its row $p$ excludes every column outside $\up p$, so it is
maximal. These rectangles are distinct, since equality of principal ideals
implies equality of their greatest elements.

Now let $A\times B$ be any maximal nonempty 1-rectangle of $M_P$. Every
element of $B$ is an upper bound for $A$. Finite repeated use of the bounded
binary-supremum property gives $p=\bigvee A$. The set of all upper bounds
of $A$ is $\up p$, so maximality forces $B=\up p$ and then $A=\down p$.
Thus there are no other maximal nonempty 1-rectangles. Finally,
$R_x\cap R_y\ne\varnothing$ is precisely \eqref{eq:boundgraph}.
\end{proof}

\begin{corollary}\label{cor:cs-certificate}
Every cs-graph on $n$ vertices has an $n\times n$ binary representation.
Recognition of cs-graphs belongs to $\mathrm{NP}$.
\end{corollary}
\begin{proof}
The representation is $M_P$ from \cref{prop:normal}. A certificate is its
$n\times n$ order relation. Reflexivity, antisymmetry and transitivity can
be checked in polynomial time. For each pair, compute its common lower and
upper bounds and check that every nonempty such set is a principal ideal or
filter. Finally check \eqref{eq:boundgraph} against the input graph.
All these operations take polynomial time; no maximal-rectangle enumeration
is needed.
\end{proof}
The certificate bound does not, by itself, give a polynomial-time algorithm
for finding an IIC order of an arbitrary cs-graph. A polynomial-time
recognizer for the hereditary subclass will follow from the construction
in \cref{sec:orientation}.

\subsection{Comparison with hereditary double-bound graphs}\label{sec:hdb}
Let $\HDB$ be the class of graphs whose every induced subgraph is a
double-bound graph of some finite poset. The representing poset may vary
with the induced subgraph. This is the hereditary representation property
studied in \cite{EraIwaiOgawaTsuchiya2003,EraOgawaTsuchiya2008}; no IIC
condition is imposed in the definition of $\HDB$.

\begin{proposition}\label{prop:hdb}
One has $\beautiful\subsetneq\HDB$.
\end{proposition}
\begin{proof}
If $G$ is beautiful, each $G[W]$ is a cs-graph and hence, by
\cref{prop:normal}, the double-bound graph of an IIC poset. Thus
$G\in\HDB$.

Every bipartite graph $H$ with bipartition $L\dotcup U$ is a
double-bound graph: take the height-at-most-two order whose strict
comparisons are $l<u$ exactly for $lu\in E(H)$. Distinct elements on the
same side cannot have both kinds of common bound, while a pair on
opposite sides has both precisely when it is comparable. The double-bound
graph is therefore $H$. Since induced subgraphs of a bipartite graph are
bipartite, every bipartite graph belongs to $\HDB$.

In particular $C_4\in\HDB$, but $C_4\notin\beautiful$ by the known
exclusion in \cite{Wunderlich2009}, also proved independently in
\cref{prop:excluded} below. The height-two representation of $C_4$ is not
IIC: its two minimal elements have two incomparable minimal common upper
bounds and no supremum. This proves strictness.
\end{proof}

The IIC requirement therefore cannot simply be dropped from the cs
normal form or from a comparison of the two hereditary classes. The
strict inclusion distinguishes the classes; the proof of
\cref{thm:main} below is independent of the earlier hereditary-double-bound
characterizations.

\subsection{A hereditary sufficient certificate}\label{sec:certificate}
\begin{lemma}\label{lem:c4-iic}
Every finite poset whose comparability graph is $C_4$-free is IIC.
\end{lemma}
\begin{proof}
Consider incomparable elements $x,y$ having common lower bounds. Two distinct
maximal elements $a,b$ of the set of common lower bounds would be
incomparable and distinct from $x,y$. The four elements $a,b,x,y$ would
induce a $K_{2,2}=C_4$ in the comparability graph. Hence that finite,
nonempty set has a unique maximal element, which is its greatest element.
For comparable $x,y$, the smaller element is already their infimum. The
argument for upper bounds is dual.
\end{proof}

\begin{proposition}[Hereditary order-matrix certificate]\label{prop:hereditary}
Suppose $P$ has chain intervals and $G=\Comp(P)$ is $C_4$-free. Then $G$
is beautiful, and \eqref{eq:hereditary-matrix} holds with the natural
vertex-to-rectangle correspondence.
\end{proposition}
\begin{proof}
For every $W\subseteq P$, the induced subposet $Q=P[W]$ has chain intervals
and comparability graph $G[W]$. This graph is still $C_4$-free, so
\cref{lem:c4-iic} makes $Q$ IIC. Comparable elements have both a common lower
bound and a common upper bound. Conversely, elements having both kinds of
bound lie in a common interval, which is a chain. Thus $\DB(Q)=\Comp(Q)$.
By \cref{prop:normal}, the order matrix of $Q$, which is $M_P[W,W]$, has
exactly the rectangles $(\down_Qq)\times(\up_Qq)$, for $q\in W$, and its
cs-graph is $G[W]$.
\end{proof}

\subsection{Short proofs of the necessary exclusions}\label{sec:exclusions}
The normal form gives a common starting point for the exclusions originally
proved in \cite[Theorem~3.6]{Wunderlich2009}.

\Needspace{8\baselineskip}
\begin{lemma}\label{lem:k4-order}
If $G$ is a $K_4$-free cs-graph, then $G=\Comp(P)$ for an IIC poset $P$
with chain intervals.
\end{lemma}
\begin{proof}
Use $G=\DB(P)$ from \cref{prop:normal}. If incomparable $x,y$ were adjacent in
$\DB(P)$, both $a=x\wedge y$ and $b=x\vee y$ would exist. The four elements
$a,x,y,b$ would be distinct and would share the lower bound $a$ and the
upper bound $b$. They would therefore form a $K_4$ in $\DB(P)$, a
contradiction. Hence $\DB(P)=\Comp(P)$. The same argument shows that an
interval of $P$ cannot contain incomparable elements.
\end{proof}

\begin{proposition}[Known exclusions, with order-theoretic proofs]\label{prop:excluded}
No member of $\F$ is a cs-graph.
\end{proposition}
\begin{proof}
Every member of $\F$ is $K_4$-free. If one were a cs-graph,
\cref{lem:k4-order} would supply an IIC order on its own vertex set whose
comparability graph is the given graph and whose intervals are chains.

\emph{Cycles.} An orientation of an induced cycle of length at least four
cannot contain a directed two-edge path: transitivity would require a chord.
Thus sources and sinks alternate. This is impossible on an odd cycle.
On $C_4$ it gives two incomparable lower elements and two incomparable upper
elements, with all four lower-to-upper comparisons. The lower pair then has
common upper bounds but no supremum, contrary to the IIC property.

\emph{Net.} Its triangle must be linearly ordered, say $a<b<c$. Let $z$ be
the leaf at $b$. If $z<b$, then $z<c$; if $b<z$, then $a<z$. Either option
forces a nonexistent edge incident with $z$.

\emph{Gem.} Let $u$ be the universal vertex and let $Q$ be the induced
$P_4$. A nonadjacent pair in $Q$ cannot lie on opposite sides of $u$ in
the order. Since the complement of $P_4$ is connected, all vertices of $Q$
lie on the same side of $u$. Reversing the order if necessary, $u$ is below
them all. A transitive orientation of $P_4$ alternates, and therefore has a
vertex $z$ with two nonadjacent predecessors in $Q$. These predecessors
both belong to $[u,z]$, contradicting the chain-interval property.

\emph{Watch.} Let $u,v$ be its diamond centres, $a,b$ its nonadjacent tips,
and $p,q$ the leaves at $u,v$, respectively. Assume $u<v$. The nonedges
$pv$ and $qu$ force $u<p$ and $q<v$. A tip $t\in\{a,b\}$ cannot satisfy
$t<u$, as that would force $t<p$; it cannot satisfy $v<t$, as that would
force $q<t$. Thus $u<t<v$ for both tips. The interval $[u,v]$ contains the
incomparable pair $a,b$, again a contradiction.
\end{proof}

\section{The structure forced by the forbidden families}\label{sec:structure}
We now turn to the sufficient direction of \cref{thm:main}. The next two
lemmas translate small forbidden configurations into closed-neighbourhood
containment.

\begin{lemma}[Nested neighbourhoods at diamond centres]\label{lem:diamond}
In a $(C_4,\gem,\watch)$-free graph, the two degree-three vertices $u,v$ of
any induced diamond satisfy
\[
 N[u]\subseteq N[v]\quad\text{or}\quad N[v]\subseteq N[u].
\]
\end{lemma}
\begin{proof}
Let $a,b$ be the nonadjacent tips. Suppose that $x$ is adjacent to $u$ but
not to $v$. If $x$ is adjacent to both tips, then $x,a,v,b$ induce a $C_4$.
If it is adjacent to exactly one tip, say $a$, then $x,a,v,b$ induce a
$P_4$ with universal vertex $u$, giving a gem. Consequently $x$ is
adjacent to neither tip. Dually, a vertex $y$ adjacent to $v$ but not to
$u$ is adjacent to neither tip.

If $N[u]$ and $N[v]$ were incomparable, both $x$ and $y$ would exist.
When $xy$ is an edge, $x,u,v,y$ induce a $C_4$. When it is not, the
original diamond together with $x,y$ induces a watch. Both are forbidden.
\end{proof}

\Needspace{9\baselineskip}
\begin{lemma}[Nesting below a common dominator]\label{lem:dominator}
In a $(C_4,\gem)$-free graph, suppose $xy\in E(G)$ and
\[
 N[x]\subseteq N[r],\qquad N[y]\subseteq N[r].
\]
Then $N[x]$ and $N[y]$ are comparable by inclusion.
\end{lemma}
\begin{proof}
Otherwise choose $a\in N[x]\setminus N[y]$ and
$b\in N[y]\setminus N[x]$. The vertices $a,x,y,b$ are distinct.
An edge $ab$ makes them an induced $C_4$; without that edge they induce a
$P_4$. In the latter case $r$ is distinct from all four vertices and adjacent
to all of them, because it dominates both closed neighbourhoods. This gives
a gem.
\end{proof}

\subsection{The graph of maximal closed neighbourhoods}\label{sec:core}
For the remainder of \cref{sec:structure,sec:orientation}, let $G$ be
$\F$-free. Choose a set $R$ containing one representative for each
\emph{distinct inclusion-maximal closed neighbourhood}, and put
\[
 H=G[R].
\]
For distinct $r,s\in R$, the sets $N[r]$ and $N[s]$ are incomparable.

\begin{proposition}\label{prop:core}
The graph $H$ is bipartite.
\end{proposition}
\begin{proof}
Suppose first that $r,s,t$ form a triangle in $H$. Since
$N[r]\nsubseteq N[s]$, there is a vertex $x$ adjacent to $r$ but not to $s$.
It cannot be adjacent to $t$: otherwise $r,t$ would be the centres of an
induced diamond with tips $s,x$, contradicting \cref{lem:diamond} and the
incomparability of $N[r],N[t]$. Hence $x$ is adjacent to exactly $r$ among
the triangle vertices.

Cyclically obtain vertices $y,z$ adjacent to exactly $s,t$, respectively,
among the triangle vertices. These three vertices are distinct and pairwise
nonadjacent. For example, an edge $xy$ would produce the induced cycle
$x,r,s,y,x$. Thus the six vertices induce a net, a contradiction. This
proves that $H$ is triangle-free.

If $H$ were not bipartite, a shortest odd cycle in $H$ would be induced
and have length at least five. Since $H$ is induced in $G$, this would be
an odd hole in $G$. Therefore $H$ is bipartite.
\end{proof}

Fix a bipartition $R=L\dotcup U$. For a vertex $v$ define its
\emph{maximal dominators} by
\begin{equation}\label{eq:dom}
 D(v)=\{r\in R:N[v]\subseteq N[r]\}.
\end{equation}
Finiteness implies $D(v)\ne\varnothing$. Moreover, $D(v)$ is a clique in
$H$. Indeed, if distinct $r,s$ dominate $v$, then
$r\in N[v]\subseteq N[s]$. As $H$ is bipartite,
\begin{equation}\label{eq:dom-size}
 |D(v)|\in\{1,2\}.
\end{equation}
We therefore obtain a partition into regions
\begin{align}
 V_r&=\{v:D(v)=\{r\}\}, &&r\in R,\label{eq:single-region}\\
 W_{rs}&=\{v:D(v)=\{r,s\}\}, &&rs\in E(H).\label{eq:double-region}
\end{align}
In particular $r\in V_r$, and a two-dominator region can occur only on an
edge of $H$. The sets $W_{rs}$ may be empty.

\subsection{Regions and their interfaces}\label{sec:interfaces}
For an edge $rs\in E(H)$, with a chosen ordering of its endpoints, define
\begin{equation}\label{eq:interfaces}
 K_{rs}=N[r]\cap N[s],\qquad
 A_{rs}=V_r\cap N[s],\qquad B_{rs}=V_s\cap N[r].
\end{equation}
For the orientation we will always choose $r\in L$, $s\in U$.

\Needspace{13\baselineskip}
\begin{lemma}\label{lem:regions}
The following properties hold.
\begin{enumerate}[label=\textnormal{(\roman*)}]
\item $K_{rs}$ is a clique and
\begin{equation}\label{eq:interface-partition}
 K_{rs}=A_{rs}\dotcup W_{rs}\dotcup B_{rs}.
\end{equation}
Every $w\in W_{rs}$ satisfies $N[w]=K_{rs}$. Thus a nonempty $W_{rs}$ is
a clique of simplicial true twins.
\item For distinct $x,y\in V_r$, adjacency is equivalent to comparability
of $N[x],N[y]$ by inclusion. Reverse inclusion, with consistent linear
tie-breaking within each true-twin class, defines a poset $T_r$ on $V_r$
whose principal ideals are chains.
\item $A_{rs}$ is a predecessor-closed chain in $T_r$, and $B_{rs}$ is a
predecessor-closed chain in $T_s$.
\item Every edge between different regions belongs to an interface $K_{rs}$.
There are no edges between distinct two-dominator regions. Every vertex of
$W_{rs}$ has neighbours outside $W_{rs}$ exactly $A_{rs}\cup B_{rs}$.
\end{enumerate}
\end{lemma}
\begin{proof}
If two vertices of $K_{rs}$ were nonadjacent, they would be the tips of a
diamond with centres $r,s$, contradicting \cref{lem:diamond}. Hence $K_{rs}$
is a clique. For $w\in W_{rs}$, domination gives $N[w]\subseteq K_{rs}$.
Since $w\in K_{rs}$ and the latter is a clique, equality holds.

For any $z\in K_{rs}$ and $t\in D(z)$, the representative $t$ is adjacent
to both $r,s$ unless it equals one of them. A third representative would
form a triangle in $H$. Consequently $D(z)\subseteq\{r,s\}$, proving
\eqref{eq:interface-partition} and~(i).

For vertices of $V_r$, adjacency implies neighbourhood nesting by
\cref{lem:dominator}. Conversely, containment of two distinct vertices'
closed neighbourhoods implies adjacency. Choose a linear order within
each equality class and refine reverse inclusion by these orders. This
gives a partial order $T_r$ with comparability graph $G[V_r]$. If
$x,y\le_{T_r}z$, then $N[z]\subseteq N[x]\cap N[y]$. Symmetry of adjacency
gives $x\in N[z]\subseteq N[y]$, so $x,y$ are adjacent or equal and
therefore comparable in $T_r$. Every principal ideal is a chain, proving~(ii).

The set $A_{rs}$ is a clique, being a subset of $K_{rs}$, and hence a chain
in $T_r$. If $x\in A_{rs}$ and $y\le_{T_r}x$, then
$N[x]\subseteq N[y]$. As $s\in N[x]$, also $s\in N[y]$, so $y\in A_{rs}$.
The argument for $B_{rs}$ is identical. This proves~(iii).

Finally, if $x\in V_r$ and $y\in V_s$ are adjacent, where $r\ne s$,
domination gives $rs\in E(H)$, $x\in N[s]$, and $y\in N[r]$. Thus both
vertices lie in $K_{rs}$. All assertions involving $W_{rs}$ follow from
$N[w]=K_{rs}$ and \eqref{eq:interface-partition}. In particular, a neighbour
of $W_{rs}$ cannot lie in a different two-dominator region. This proves~(iv).
\end{proof}

The single-dominator regions therefore have local orders with chain
intervals. Their interfaces join predecessor-closed chains, with a possible
simplicial true-twin clique between them. Arbitrary gluing of beautiful
graphs would not provide this compatibility information.

\section{The global orientation}\label{sec:orientation}
We orient all edges of the $\F$-free graph $G$ using the bipartition
$R=L\dotcup U$ and the regions above.
\begin{enumerate}[label=\textnormal{(\arabic*)}]
\item Inside $V_r$, for $r\in L$, use $T_r$: orient from larger to smaller
closed neighbourhoods, with the chosen tie-breaking.
\item Inside $V_s$, for $s\in U$, reverse $T_s$: orient from smaller to
larger closed neighbourhoods, reversing the same tie-breaking.
\item Order each $W_{rs}$ linearly. For $r\in L$, $s\in U$, orient the
interface clique in the block order
\begin{equation}\label{eq:interface-orientation}
 A_{rs}\ \longrightarrow\ W_{rs}\ \longrightarrow\ B_{rs},
\end{equation}
including every edge from $A_{rs}$ to $B_{rs}$. The orders within $A_{rs}$
and $B_{rs}$ are those already fixed in their regions.
\end{enumerate}
The prescription applies also when $W_{rs}=\varnothing$.
By \cref{lem:regions}, every edge receives one consistent orientation.
An internal region edge keeps its fixed orientation wherever it occurs
in an interface; a cross-region edge identifies its two representatives,
and an edge incident with a two-dominator region belongs to its single
interface.

\begin{lemma}\label{lem:transitive}
The resulting orientation is transitive.
\end{lemma}
\begin{proof}
Consider a directed two-edge path $x\to y\to z$. If $y\in W_{rs}$, all
its neighbours belong to $K_{rs}$, which is linearly oriented. Thus the
edge $x\to z$ is present.

If $y\in V_r$ with $r\in L$, every incoming edge at $y$ comes from $V_r$:
all cross-region edges leave a lower-side region. Hence $x\in V_r$ and
$N[y]\subseteq N[x]$. If also $z\in V_r$, transitivity follows from $T_r$.
Otherwise $z\in N[y]\subseteq N[x]$, so $xz$ is an edge, and its
cross-region orientation is $x\to z$.

If $y\in V_s$ with $s\in U$, every outgoing edge stays in $V_s$. Thus
$z\in V_s$ and $N[y]\subseteq N[z]$. If $x\in V_s$, use the local order.
Otherwise $x\in N[y]\subseteq N[z]$, and the cross-region rule again
gives $x\to z$.
\end{proof}

A transitive orientation of a simple graph is acyclic: repeated transitivity
along a directed cycle would force both directions of an edge. Its reflexive
closure is therefore a poset $P$, and
\begin{equation}\label{eq:comp-global}
 \Comp(P)=G.
\end{equation}

\begin{lemma}\label{lem:chain-intervals}
Every interval of $P$ is a chain.
\end{lemma}
\begin{proof}
Directed paths can pass through at most one lower-side region $V_r$, then
possibly its two-dominator region $W_{rs}$, and then at most one upper-side
region $V_s$. Some parts may be absent. A path cannot move between
lower-side regions, leave an upper-side region, or move between distinct
two-dominator regions. These assertions follow from
\cref{lem:regions}\textnormal{(iv)} and the direction of every cross-region edge.

Consider an interval $[x,z]_P$. If both endpoints lie in the same $V_r$,
a directed path between them cannot leave and re-enter that region.
The entire interval therefore stays in $V_r$. It is an interval of $T_r$
or its reverse, and hence is a chain by
\cref{lem:regions}\textnormal{(ii)}.

If an endpoint belongs to $W_{rs}$, every element of the interval is equal
or adjacent to that endpoint. Its closed neighbourhood is $K_{rs}$, so the
interval lies in the linearly oriented clique $K_{rs}$.

It remains that $x\in V_r$ and $z\in V_s$ are in distinct single-dominator
regions. Necessarily $r\in L$, $s\in U$, and $rs\in E(H)$.
Every intermediate element lies in $V_r$, $W_{rs}$, or $V_s$, by the
restricted path shape. An intermediate element in $V_r$ is adjacent to $z$
by transitivity and hence lies in $A_{rs}$. Dually, one in $V_s$ lies in
$B_{rs}$. The endpoints also belong to these respective sets. Thus the
whole interval lies in the linearly oriented $K_{rs}$ and is a chain.
\end{proof}

\begin{proof}[Proof of \cref{thm:main}]
If $G$ is beautiful, none of its induced subgraphs can be a member of $\F$,
by \cref{prop:excluded}. This proves \textnormal{(i)}$\Rightarrow$\textnormal{(ii)}.
For an $\F$-free graph, the construction above and
\cref{lem:transitive,lem:chain-intervals} give a poset $P$ with chain intervals
and comparability graph $G$. Since $G$ is $C_4$-free, this proves
\textnormal{(ii)}$\Rightarrow$\textnormal{(iii)}.
Finally, \cref{prop:hereditary} proves
\textnormal{(iii)}$\Rightarrow$\textnormal{(i)} together with the
simultaneous principal-submatrix representations.

Statement \textnormal{(iv)} implies \textnormal{(ii)}, since a Berge graph
has no odd hole. Conversely, an $\F$-free graph also has no odd antihole:
the complement of $C_5$ is $C_5$, and for $m\ge7$ the cyclically labelled
vertices $0,1,3,4$ induce $2K_2$ in $C_m$, hence a $C_4$ in its complement.
Thus \textnormal{(ii)} implies \textnormal{(iv)}. The empty graph is
immediate under the stated convention.
\end{proof}

\section{Consequences and examples}\label{sec:consequences}

\subsection{Position among Berge graph classes}\label{sec:berge}
The original Berge containment is recovered directly from the
characterization. No appeal to the strong perfect graph theorem is needed
for this observation.

\begin{corollary}\label{cor:berge}
Every beautiful graph is Berge. Moreover,
\begin{equation}\label{eq:proper-inclusions}
 \mathsf{Bipartite}\cap\Forb(C_4)
 \ \subsetneq\ \beautiful
 \ \subsetneq\ \mathsf{Comparability}\cap\Forb(C_4).
\end{equation}
\end{corollary}
\begin{proof}
The Berge assertion is \cref{thm:main}\textnormal{(iv)}.

A $C_4$-free bipartite graph avoids $\F$: it has no odd cycle, and each
of gem, net, and watch contains a triangle. The first inclusion follows
from \cref{thm:main}, and $K_3$ makes it proper. The second inclusion is
part of that theorem. To see that it is proper, the gem is $C_4$-free and
is a comparability graph: for its path $a,b,c,d$ and universal vertex $u$,
take $u$ below all four path vertices, with $a<b$, $c<b$, and $c<d$.
This is a transitive orientation of the gem. The gem is not beautiful.
\end{proof}

Theorem~\ref{thm:main} also distinguishes beautiful graphs from familiar
Berge subclasses in \cite{BrandstadtLeSpinrad1999,ISGCI}.
\Cref{tab:separations} gives explicit witnesses of incomparability with
five such classes; these comparisons are not inferred merely from the
absence of a class name in a catalogue. The class distinctions motivate
the classification, whereas the new assertion is the exact completeness
of $\F$.

\begin{table}[htbp]
\centering
\begin{tabular}{@{}lcc@{}}
\toprule
Class $\mathcal C$ & $\beautiful\setminus\mathcal C$
                  & $\mathcal C\setminus\beautiful$\\
\midrule
Chordal                & $C_6$ & $\gem$\\
Ptolemaic              & $C_6$ & $\net$\\
Distance-hereditary    & $C_6$ & $C_4$\\
Weakly chordal         & $C_6$ & $C_4$\\
Cographs               & $P_4$ & $C_4$\\
\bottomrule
\end{tabular}
\caption{Separating examples for selected established Berge classes.
Each column contains a graph in the indicated set difference.}
\label{tab:separations}
\end{table}

Here chordal graphs have no holes, weakly chordal graphs and their
complements have no induced cycles of length at least five, and cographs
are $P_4$-free. Ptolemaic graphs are the chordal distance-hereditary graphs
\cite{Howorka1981,BrandstadtLeSpinrad1999}; a graph is distance-hereditary
when each connected induced subgraph preserves distances. The gem and net
are chordal, and the net is distance-hereditary: a connected induced
subgraph with at least two vertices has a clique of surviving triangle
vertices and any surviving leaves attach only to their original triangle vertices, so distances do
not change. The graph $C_4$ is distance-hereditary, weakly chordal, and
$P_4$-free, but is not beautiful. The beautiful graph $C_6$ belongs to
none of the first four classes: in particular, deleting one vertex
leaves a $P_5$ whose endpoints have distance four rather than two.
Finally $P_4$ is beautiful but not a cograph. These observations verify
all entries of \cref{tab:separations}.

\subsection{The minimal forbidden induced subgraphs}
\begin{corollary}\label{cor:minimal}
The minimal forbidden induced subgraphs for $\beautiful$ are exactly
\[
 C_4,\quad \gem,\quad \net,\quad \watch,\quad C_5,C_7,C_9,\ldots.
\]
Thus four individual graphs and one infinite cycle family suffice, whereas
a finite list of individual forbidden graphs cannot suffice.
\end{corollary}
\begin{proof}
The members of $\F$ form an induced-subgraph antichain. The three graphs
gem, net and watch have no induced cycle of length at least four and do
not contain one another as induced subgraphs; the latter statement is
checked directly from their definitions. A proper induced subgraph of a
cycle is a disjoint union of paths, which contains none of $\F$.
Consequently every proper induced subgraph of a member of $\F$ is
$\F$-free, hence beautiful by \cref{thm:main}. Conversely, every
non-beautiful graph contains a member of $\F$. In particular, all odd
holes are distinct minimal obstructions.
\end{proof}

For example, within the class of diamond-free graphs the gem and watch
are already absent, since each contains a diamond. \Cref{thm:main} therefore
specializes to
\[
 G\text{ diamond-free}\quad\Longrightarrow\quad
 \bigl(G\text{ beautiful}\ \Longleftrightarrow\
 G\text{ is }(C_4,\net,\text{odd holes})\text{-free}\bigr).
\]

\subsection{Recognition and positive certificates}\label{sec:algorithm}
\begin{corollary}\label{cor:recognition}
Beautiful graphs can be recognized in $O(n^4)$ time and $O(n^2)$ space,
using a Boolean-matrix implementation. On acceptance, the algorithm
outputs an $n\times n$ hereditary order-matrix certificate.
\end{corollary}
\begin{proof}
The following procedure makes both completeness and soundness explicit.
\begin{enumerate}[label=\arabic*.]
\item Check induced-$C_4$-freeness. If a $C_4$ is found, reject.
\item Compute all closed-neighbourhood containments, select $R$, and
bipartition $H=G[R]$. Reject if $H$ is not bipartite.
\item Compute $D(v)$ and the regions. Reject if a support is empty, has more
than two elements, or is a two-element nonedge of $H$.
\item Form the orientation of \cref{sec:orientation}, using fixed vertex
indices to break ties. Reject if a required containment or a permitted
cross-region type is missing.
\item Independently verify that the reflexive relation is a partial order,
that its comparability graph is exactly $G$, and that no nonadjacent pair
has both a common predecessor and a common successor. Accept only if all
checks pass, and output the order matrix.
\end{enumerate}
Every beautiful input satisfies the structural lemmas, so the construction
and the final verification succeed by \cref{thm:main}. Conversely, an
accepted relation is a chain-interval order with $C_4$-free comparability
graph $G$, so \cref{prop:hereditary} guarantees beautifulness. Soundness
therefore does not rely on taking unverified intermediate structure on trust.

A direct test over four-vertex subsets uses $O(n^4)$ time. All pairwise
neighbourhood containments take $O(n^3)$ time and $O(n^2)$ space. The
remaining construction and order checks fit within these bounds; common
predecessors and successors can be tested by Boolean row and column
intersections. The total space remains $O(n^2)$.
\end{proof}
In particular, this recognizer does not enumerate odd holes and does not
enumerate maximal rectangles. It gives a positive certificate on acceptance;
we do not claim that every rejection outputs a forbidden induced subgraph.

The orientation also supplies a direct optimal colouring. Colour a vertex
by the maximum number of elements in a chain ending at it. Comparable
vertices receive different colours, and the number of colours is the
height of the order, equal to the clique number of its comparability graph.
The same argument applies to every induced subgraph. Thus beautiful graphs
are perfect by a direct order-theoretic argument.

\subsection{The corrected \texorpdfstring{$K_4$}{K4}-free statement}\label{sec:k4}
\begin{corollary}\label{cor:k4}
For a $K_4$-free graph $G$,
\[
 G\text{ is beautiful}\quad\Longleftrightarrow\quad
 G\text{ is a }C_4\text{-free cs-graph}.
\]
\end{corollary}
\begin{proof}
The forward implication follows from the definition and
\cref{prop:excluded}. Conversely, a $K_4$-free cs-graph has a
chain-interval comparability order by \cref{lem:k4-order}. If it is also
$C_4$-free, \cref{prop:hereditary} applies.
\end{proof}

The added $C_4$-freeness is essential. The stronger equality between
$K_4$-free cs-graphs and $K_4$-free beautiful graphs was conjectured in
\cite[Section~3]{Wunderlich2008}; the following matrix supplies a counterexample.

\begin{example}\label{ex:wheel}
Let
\[
 M=\begin{pmatrix}
 1&0&1&1\\
 0&1&1&1\\
 0&0&1&0\\
 0&0&0&1
 \end{pmatrix}.
\]
Its maximal nonempty 1-rectangles are exactly
\begin{align*}
 R_1&=\{1\}\times\{1,3,4\}, &
 R_2&=\{2\}\times\{2,3,4\},\\
 R_3&=\{1,2\}\times\{3,4\}, &
 R_4&=\{1,2,3\}\times\{3\},\\
 R_5&=\{1,2,4\}\times\{4\}. &&
\end{align*}
A rectangle containing row~3 extends maximally to $R_4$; one containing
row~4 extends to $R_5$. Without either row, the three possible nonempty
row sets give $R_1,R_2,R_3$. This proves completeness of the list.
The rectangles $R_1,R_4,R_2,R_5$ form an induced $C_4$ in that order,
while $R_3$ intersects all four. Hence
\[
 \cs(M)\cong K_1\vee C_4,
\]
where $\vee$ denotes the join. This graph has clique number three but is
not beautiful.
\end{example}
Consequently $\csg$ is not hereditary and cannot itself have a
forbidden-induced-subgraph characterization. This is why
\cref{prop:normal} is an order-theoretic statement, whereas
\cref{thm:main} is a forbidden-family statement.

\Needspace{16\baselineskip}
\subsection{An example with a diamond}\label{sec:example}
On vertices $u,v,a,b,x,y$, take the edge set
\[
 \{uv,ua,ub,va,vb,ux,ay\}.
\]
The maximal-neighbourhood representatives are $u,a$. Their regions are
$V_u=\{u,v,b,x\}$, $V_a=\{a,y\}$, and $W_{ua}=\varnothing$; the
interface is $K_{ua}=\{u,v,a\}$. Put $u$ on the lower side and $a$ on the
upper side. The construction gives the order generated by
\[
 u<v<b,\qquad v<a,\qquad u<x,\qquad y<a.
\]
In the vertex order $(u,v,a,b,x,y)$, its matrix is
\[
 M_P=\begin{pmatrix}
 1&1&1&1&1&0\\
 0&1&1&1&0&0\\
 0&0&1&0&0&0\\
 0&0&0&1&0&0\\
 0&0&0&0&1&0\\
 0&0&1&0&0&1
 \end{pmatrix}.
\]
The graph-theoretic diamond has centres $u,v$, but its two tips lie above
the chain $u<v$. They do not occur between common lower and upper bounds.
Thus a graph diamond need not be a non-chain order interval. Every
principal submatrix of this $M_P$ represents the corresponding induced
subgraph.

\section{Concluding remarks}\label{sec:conclusion}
The main result completes the forbidden-family description of the
beautiful Berge class introduced in \cite{Wunderlich2009}. Its necessity
was known; the maximal-neighbourhood decomposition proves sufficiency
without a clique-size restriction. The representation tools connect this
classification to established theories: for general cs-graphs, an IIC
poset and its classical double-bound graph record exactly when maximal
rectangles share both a row and a column. For beautiful graphs, the forbidden families force a
stronger order whose intervals are chains; in the $C_4$-free setting this
order persists as a valid matrix certificate under every vertex restriction.
The maximal-neighbourhood graph supplies the bipartition needed to orient
all regions consistently, including those containing diamonds and large
cliques.

The construction leaves algorithmic refinements to investigate, such as a
faster recognizer with an explicit forbidden-subgraph certificate on every
negative instance. It also suggests studying the minimum matrix dimensions
needed to represent a beautiful graph: the uniform $n\times n$ bound need
not be optimal. Neither question is needed for the characterization proved
here.

\section*{Disclosure of AI assistance}
OpenAI ChatGPT was used as an assistance tool for drafting, checks, and locating related literature.  
The author reviewed, revised, and takes responsibility for the mathematical statements, arguments, citations, and final contents of the manuscript.

\renewcommand{\bibfont}{\small}
\setlength{\bibsep}{4pt}
\bibliographystyle{plainnat}
\bibliography{references}
\end{document}